\documentclass[12pt,final]{amsart}
\usepackage{amssymb,euscript}
\usepackage{amscd}
\usepackage[notref,notcite]{showkeys}
\usepackage{latexsym,todonotes}
\usepackage{epsfig,pxfonts}
\usepackage{amsfonts}

\usepackage{xr,xr-hyper}
\usepackage{bbm,ifpdf}
\ifpdf
  \usepackage{pdfsync,hyperref}
\fi
\usepackage[margin=3cm,footskip=25pt,headheight=21pt]{geometry}
\usepackage{fancyhdr,mathrsfs,nicefrac}
\newtheorem{theorem}{Theorem}[section]
\newtheorem{lemma}[theorem]{Lemma}
\newtheorem{proposition}[theorem]{Proposition}
\newtheorem{corollary}[theorem]{Corollary}

\newtheorem{itheorem}{Theorem}

{
\theoremstyle{plain}
\newtheorem{definition}[theorem]{Definition}

}

\newcommand{\nc}{\newcommand}
\nc{\cat}{\mathcal{V}}

\newcommand{\arxiv}[1]{\href{http://arxiv.org/abs/#1}{\tt arXiv:\nolinkurl{#1}}}

\newcommand{\id}{\operatorname{id}}

\renewcommand{\dim}{\operatorname{dim}}

\newcommand{\doubletilde}[1]{
  \tilde{{\tilde{#1}}}}
\newcounter{subeqn}
\renewcommand{\thesubeqn}{\theequation\alph{subeqn}}
\newcommand{\subeqn}{%
  \refstepcounter{subeqn}
  \tag{\thesubeqn}
}
\makeatletter
\@addtoreset{subeqn}{equation}
\newcommand{\newseq}{%
  \refstepcounter{equation}
}

\newcommand{\Bi}{\mathbf{i}}

\newcommand{\K}{\mathbf{k}}

\nc{\eE}{\EuScript{E}}
\nc{\eF}{\EuScript{F}}

\newcommand{\Z}{\mathbb{Z}}

\newcommand{\R}{\mathbb{R}}

\newcommand{\C}{\mathbb{C}}

\nc{\ep}{\epsilon}

\newcommand{\la}{\leftarrow}

\nc{\Bv}{\mathbf{v}}
  \nc{\Bw}{\mathbf{w}}
\nc{\coho}{\EuScript{G}}
\nc{\sllhat}{\mathfrak{\widehat{sl}}_\ell}
\nc{\slehat}{\mathfrak{\widehat{sl}}_e}
\nc{\glehat}{\mathfrak{\widehat{gl}}_e}

\renewcommand{\la}{\lambda}

\nc{\Tr}{\operatorname{Tr}}
\nc{\tU}{\mathcal{U}}

\newcommand{\al}{\alpha}

\newcommand{\Hom}{\operatorname{Hom}}
\newcommand{\A}{\mathcal A}

\nc{\lift}{\gamma}

\newcommand{\cO}{\mathcal{O}}

\newcommand{\Ext}{\operatorname{Ext}}

\newcommand{\excise}[1]{}

\nc{\wela}{\EuScript{X}}
\nc{\rola}{\EuScript{Y}}

\newcommand{\End}{\operatorname{End}}

\newcommand{\fM}{\mathfrak{M}}

\newcommand{\fg}{\mathfrak{g}}

\newcommand{\mmod}{\operatorname{-mod}}
\newcommand{\dgmod}{\operatorname{-dg-mod}}

\newcommand{\bla}{{\underline{\boldsymbol{\la}}}}

\newcommand{\thetitle}{Centers of KLR algebras and cohomology rings of quiver varieties}

\begin{document}

\renewcommand{\theitheorem}{\Alph{itheorem}}
\usetikzlibrary{decorations.pathreplacing,backgrounds,decorations.markings}
\tikzset{wei/.style={draw=red,double=red!40!white,double distance=1.5pt,thin}}

\noindent {\Large \bf 
\thetitle}
\bigskip\\
{\bf Ben Webster}\footnote{Supported by the NSF under Grant
  DMS-1151473 and the Alfred P. Sloan Foundation}\\
Department of Mathematics,  University of Virginia, Charlottesville, VA
\bigskip\\
{\small
\begin{quote}
\noindent {\em Abstract.} 
Attached to a weight space in an integrable highest weight
representation of a simply-laced Kac-Moody algebra $\mathfrak{g}$, there are two
natural commutative algebras: the
cohomology ring of a quiver variety and the center of
a cyclotomic KLR algebra.  In this note, we describe a natural geometric
map between these algebras in terms of quantum coherent
sheaves on quiver varieties.

The cohomology ring of an algebraic symplectic variety can be
interpreted as the Hochschild cohomology of a quantization of this variety in the sense of
Bezrukavnikov and Kaledin.  On the other hand, cyclotomic KLR algebras
appear as Ext-algebras of certain particular sheaves, and thus its
center receives a canonical map from the Hochschild cohomology of the
category.  We show that this map is an isomorphism in finite type, and injective in general.  We further note that the
Kirwan surjectivity theorem for quivers of finite type is an easy corollary of these results.

The most important property of this map is its compatibility
with actions of the current algebra on both the cohomology of quiver
varieties and on the Hochschild cohomology of any category with a
categorical action of $\mathfrak{g}$.  The structure of these current
algebra actions
allow us to show the desired results.
\end{quote}
} 
\medskip Let $\fg$ be a simply-laced Kac-Moody algebra.  For each pair of
weights $\la,\mu$ such that $\mu\leq \la$, we have a quiver variety
$\fM^\la_\mu$.  Nakajima \cite{Nak98} has shown that the middle degree
cohomology of $\fM^\la_\mu$ is isomorphic to the $\mu$-weight of the
representation with highest weight $\la$, making this variety a
geometric avatar of the weight space.  On the other hand, we also have
a cyclotomic KLR algebra $R^\la_\mu$, which provides a
categorification of this weight space, in the sense that
$K^0(R^\la_\mu)$ is isomorphic to an integral form of this weight
space, with the classes of indecomposable projectives matching the
canonical basis.

Attached to these objects, we have a pair of commutative algebras: the
cohomology ring $H^*(\fM^\la_\mu)$ of the quiver variety, and the
center $Z(R^\la_\mu)$ of the KLR algebra.  It is clear from various
analogies that these algebras should be closely related.  For example,
both have actions of a current algebra, and both are isomorphic to
dual Weyl modules for this current algebra in finite type.  Also, in
type A, explicit calculations show that they are isomorphic as algebras.

Our goal is to show that these algebras are isomorphic for any ADE
type quiver variety and to give a natural, geometric isomorphism
between them.  For other simply-laced types, we show that there is an
injection $H^*(\fM^\la_\mu)\hookrightarrow Z(R^\la_\mu)$.

Constructing this map requires the theory of quantizations of quiver
varieties.  We'll follow the notation
of \cite{Webcatq} throughout.  
In that work, we concentrated our interest in a quantization $\A_\mu$ of the
structure sheaf of $\fM^\la_\mu$ compatible with its symplectic
structure.  These quantizations are indexed by {\bf periods}, and
we'll specialize throughout to the case where the period is integral.  

There is a categorical action of the Lie algebra
$\fg$ on the coherent modules over the algebra  $\A_\mu$ for all the
different weights $\mu$ appearing in the weight decomposition of the
representation with highest weight $\la$ \cite{Webcatq}.
Inside the derived category of $\A_\mu$-modules, there's a natural
subcategory $\mathcal{C}$, which is the
smallest invariant subcategory containing the constant sheaf on
$\fM^\la_\la$.  This can also be geometrically characterized as the
subcategory of objects with compact support.  In \cite[Th. \ref{O-main1}]{Webqui}, we
showed an equivalence of dg-categories $\mathcal{C}\cong
R^\la_\mu\dgmod$, where $R^\la_\mu$ is the cyclotomic KLR algebra
attached to the same Lie algebra and pair of weights.  

We have a natural pullback map on the Hochschild cohomology of
$\A_\mu$-modules to that of $\mathcal{C}$.  Put in more down-to-earth terms, this gives a map
$\phi\colon H^*(\fM^\la_\mu)\to Z(R^\la_\mu)$, realizing the former as the
Hochschild homology of the category of $\A$-modules, as in \cite[\S
5.4]{BLPWgco}, and the latter as the Ext-algebra of a semi-simple
object in $\mathcal{C}$.

The main result of this note is:
\begin{itheorem}[\mbox{Theorem \ref{thm:final}}]\label{main}
  If $\fg$ is type ADE, the map $\phi$ is an algebra
  isomorphism.  If $\fg$ is of general simply-laced type, then $\phi$
  is injective.
\end{itheorem}

The key fact which allows us to prove this is that $\phi$ is
compatible with an action of the current algebra $\fg[t]$ on the
cohomology rings $ H^*(\fM^\la_\mu)$ was defined by Varagnolo
\cite{Varyang} and on the center $ Z(R^\la_\mu)$ by the author, Beliakova, Habiro
and Lauda \cite{BHLW} and Shan, Vasserot and Varagnolo
\cite{SVVcenter}, independently.  Once we know that $\phi$ commutes
with the current algebra action, the isomorphism of both source and
target with a dual Weyl module shows that $\phi$ is an isomorphism.

\section{Categorifications and current algebras}
\label{sec:categ-curr-algebr}
  First, we will need some general background on categorical actions.  

\subsection{The 2-category \texorpdfstring{$\tU$}{U}}
\label{sec:2-category-tu}
Fix an oriented graph $\Gamma$, which we will assume is
simply-laced.  
 The object of interest for this subsection is a strict 2-category
$\tU$, due to Khovanov and Lauda~\cite{KLIII}. We will give a more compact definition of this category, shown to be
equivalent to that of earlier literature such as \cite{KLIII,CaLa,Webmerged} in a recent paper of
Brundan \cite{Brundandef}. 

In order to define it, we will need to
define a class of diagrams.  Consider the set of diagrams in the
horizontal strip $\R\times [0,1]$ composed of embedded oriented
curves, whose endpoints lie on distinct points of $\R\times\{0\}$ and
$\R\times \{1\}$.  At each point, projection to the $y$-axis must
locally be a diffeomorphism, unless at that point it looks like one of
the diagrams:
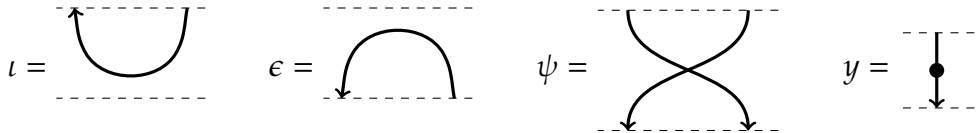
\[
\iota=\tikz[baseline,very thick,scale=3]{\draw[->] (.25,.3) to
  [out=-100, in=60]
(.2,.1)
  to[out=-120,in=-60] (-.2,.1) to [out=120,in=-80] 
(-.25,.3);
  \draw[thin,dashed] (.33,.3) -- (-.33,.3); \draw[thin,dashed] (.33,-.1) -- (-.33,-.1);}
\qquad 
\ep=\tikz[baseline,very thick,scale=3]{\draw[->] (.25,-.1) to
  [out=100, in=-60]   
 (.2,.1)
  to[out=120,in=60]
  (-.2,.1)  to [out=-120,in=80] 
  (-.25,-.1);
\draw[thin,dashed] (.33,-.1) -- (-.33,-.1); \draw[thin,dashed] (.33,.3) -- (-.33,.3); }
\qquad 
\psi=
\tikz[baseline=-2pt,very thick,scale=4]{\draw[->] (.2,.2)  to [out=-90,in=90]
(-.2,-.2) ; \draw[<-]
  (.2,-.2) to[out=90,in=-90] 
(-.2,.2); 
\draw[thin,dashed] (.3,-.2) -- (-.3,-.2); \draw[thin,dashed] (.3,.2) -- (-.3,.2);}
\qquad 
 y=\tikz[baseline=-2pt,very thick,scale=2.5]{\draw[->]
  (0,.2) -- (0,-.2) 
  node[midway,circle,fill=black,inner
  sep=2pt]{}; 
\draw[thin,dashed] (.2,-.2) -- (-.2,-.2); \draw[thin,dashed] (.2,.2) -- (-.2,.2);}\]
 
We'll consider
labelings of the components of these diagrams
by elements of $\Gamma$.  The {\bf top} of such a diagram is the sequence
where we read off the label of each of the endpoints on $\R\times \{1\}$ in
order  from left to right, taking them with positive sign if the
curve is oriented upward there, and a negative sign if it is oriented
downward.  The {\bf bottom} is defined similarly with the endpoints on
$\R\times \{0\}$.  The {\bf vertical composition} $ab$ of two diagrams where
the bottom of $a$ matches the top of $b$ is the stacking of $a$ on top
of $b$ and then scaling the $y$-coordinate by $1/2$ to lie again in
$\R\times [0,1]$.  The horizontal composition of two diagrams $a\circ
b$ places $a$ to the {\it right} of $b$ in the plane, and thus has the
effect of concatenating their tops and bottoms in the opposite of the usual order.

 We let
$\la^i=\al_i^\vee(\la)$ for any weight $\la$.
\begin{definition}
  Let $\doubletilde{\tU}$ be the strict 2-category where
 \begin{itemize}
  \item the set of objects is the weight lattice of the Kac-Moody
    algebra $\fg_\Gamma$.  
\item 1-morphisms $\mu\to \nu$ are sequences
  $\Bi=(i_1,\dots, i_m)$ with each $i_j \in \pm
  \Gamma$, which we interpret as a list of simple roots and their
  negatives such that $\mu +\sum_{j=1}^m\al_{i_j}=\nu$.  Composition
  is given by concatenation.
\item 2-morphisms $h\to h'$ between sequences are $\K$-linear combinations
  of diagrams of the type defined above with $h$ as bottom
  and $h'$ as top.    
\end{itemize}
\end{definition}
Since the underlying objects in $\doubletilde{\tU}$ are fixed for
any 2-morphism, we incorporate them into the diagram by labeling each
region of the place with $\mu$ at the far left, $\nu$ at the far
right, and intermediate regions are labeled by the rule \[  \tikz[baseline,very thick]{
\draw[postaction={decorate,decoration={markings,
    mark=at position .5 with {\arrow[scale=1.3]{<}}}}] (0,-.5) -- node[below,at start]{$i$}  (0,.5);
\node at (-1,0) {$\mu$};
\node at (1,.05) {$\mu-\al_i$};.
}.\]
We'll typically use $\eE_i$ to denote the 1-morphism $(i)$ (leaving
the labeling of regions implicit) and $\eF_i$ to denote $(-i)$.

We can define a {\bf degree} function on diagrams.  The degrees are
given on elementary diagrams by \[
  \deg\tikz[baseline,very thick,scale=1.5]{\draw[->] (.2,.3) --
    (-.2,-.1) node[at end,below, scale=.8]{$i$}; \draw[<-] (.2,-.1) --
    (-.2,.3) node[at start,below,scale=.8]{$j$};}
  =
  \begin{cases}
    -2 & i=j\\
    1 & i \leftrightarrow j\\
    0 & i \nleftrightarrow j \\
  \end{cases}
\qquad \deg\tikz[baseline,very
  thick,->,scale=1.5]{\draw (0,.3) -- (0,-.1) node[at
    end,below,scale=.8]{$i$} node[midway,circle,fill=black,inner
    sep=2pt]{};}=2 \]
  \[
  \deg\tikz[baseline,very thick,scale=1.5]{\draw[->] (.2,.1)
    to[out=-120,in=-60] node[at end,above left,scale=.8]{$i$} (-.2,.1)
    ;\node[scale=.8] at (0,.3){$\la$};} =\langle\la,\al_i\rangle-1
  \qquad \deg\tikz[baseline,very
  thick,scale=1.5]{\draw[->] (.2,.1) to[out=120,in=60] node[at
    end,below left,scale=.8]{$i$} (-.2,.1);\node[scale=.8] at
    (0,-.1){$\la$};} =-\langle\la,\al_i\rangle-1.
  \]
For a general diagram, we sum together the degrees of the elementary
diagrams it is constructed from.  This defines a grading on the
2-morphism spaces of $\doubletilde{\tU}$.

Consider the polynomials \[Q_{ij}(u,v)=
(u-v)^{\#\{ j\to i\}}
(v-u) ^{\#\{ i\to j\}} \]
\begin{definition}
Let $\tU$ be the quotient of $\doubletilde{\tU}$ by the following
relations on 2-morphisms:
\begin{itemize}
\item $\ep$ and $\iota$  are the units and counits of an adjunction,
  i.e. critical points can cancel.

\item the endomorphisms of words only using $\eF_i$ (or by duality only $\eE_i$'s) satisfy the relations of the {\bf quiver Hecke algebra} $R$.\newseq
\begin{equation*}\subeqn\label{first-QH}
    \begin{tikzpicture}[scale=.7,baseline]
      \draw[very thick,postaction={decorate,decoration={markings,
    mark=at position .2 with {\arrow[scale=1.3]{<}}}}](-4,0) +(-1,-1) -- +(1,1) node[below,at start]
      {$i$}; \draw[very thick,postaction={decorate,decoration={markings,
    mark=at position .2 with {\arrow[scale=1.3]{<}}}}](-4,0) +(1,-1) -- +(-1,1) node[below,at
      start] {$j$}; \fill (-4.5,.5) circle (3pt);
      \node at (-2,0){=}; \draw[very thick,postaction={decorate,decoration={markings,
    mark=at position .8 with {\arrow[scale=1.3]{<}}}}](0,0) +(-1,-1) -- +(1,1)
      node[below,at start] {$i$}; \draw[very thick,postaction={decorate,decoration={markings,
    mark=at position .8 with {\arrow[scale=1.3]{<}}}}](0,0) +(1,-1) --
      +(-1,1) node[below,at start] {$j$}; \fill (.5,-.5) circle (3pt);
   \end{tikzpicture}
 \qquad 
    \begin{tikzpicture}[scale=.7,baseline]
      \draw[very thick,postaction={decorate,decoration={markings,
    mark=at position .2 with {\arrow[scale=1.3]{<}}}}](-4,0) +(-1,-1) -- +(1,1) node[below,at start]
      {$i$}; \draw[very thick,postaction={decorate,decoration={markings,
    mark=at position .2 with {\arrow[scale=1.3]{<}}}}](-4,0) +(1,-1) -- +(-1,1) node[below,at
      start] {$j$}; \fill (-3.5,.5) circle (3pt);
      \node at (-2,0){=}; \draw[very thick,postaction={decorate,decoration={markings,
    mark=at position .8 with {\arrow[scale=1.3]{<}}}}](0,0) +(-1,-1) -- +(1,1)
      node[below,at start] {$i$}; \draw[very thick,postaction={decorate,decoration={markings,
    mark=at position .8 with {\arrow[scale=1.3]{<}}}}](0,0) +(1,-1) --
      +(-1,1) node[below,at start] {$j$}; \fill (-.5,-.5) circle (3pt);
      \node at (3.5,0){unless $i=j$};
    \end{tikzpicture}
  \end{equation*}
\begin{equation*}\subeqn\label{nilHecke-1}
    \begin{tikzpicture}[scale=.8,baseline]
      \draw[very thick,postaction={decorate,decoration={markings,
    mark=at position .2 with {\arrow[scale=1.3]{<}}}}](-4,0) +(-1,-1) -- +(1,1) node[below,at start]
      {$i$}; \draw[very thick,postaction={decorate,decoration={markings,
    mark=at position .2 with {\arrow[scale=1.3]{<}}}}](-4,0) +(1,-1) -- +(-1,1) node[below,at
      start] {$i$}; \fill (-4.5,.5) circle (3pt);
      \node at (-2,0){$-$}; \draw[very thick,postaction={decorate,decoration={markings,
    mark=at position .8 with {\arrow[scale=1.3]{<}}}}](0,0) +(-1,-1) -- +(1,1)
      node[below,at start] {$i$}; \draw[very thick,postaction={decorate,decoration={markings,
    mark=at position .8 with {\arrow[scale=1.3]{<}}}}](0,0) +(1,-1) --
      +(-1,1) node[below,at start] {$i$}; \fill (.5,-.5) circle (3pt);
      \node at (1.8,0){$=$}; 
    \end{tikzpicture}\,\,
    \begin{tikzpicture}[scale=.8,baseline]
      \draw[very thick,postaction={decorate,decoration={markings,
    mark=at position .8 with {\arrow[scale=1.3]{<}}}}](-4,0) +(-1,-1) -- +(1,1) node[below,at start]
      {$i$}; \draw[very thick,postaction={decorate,decoration={markings,
    mark=at position .8 with {\arrow[scale=1.3]{<}}}}](-4,0) +(1,-1) -- +(-1,1) node[below,at
      start] {$i$}; \fill (-4.5,-.5) circle (3pt);
      \node at (-2,0){$-$}; \draw[very thick,postaction={decorate,decoration={markings,
    mark=at position .2 with {\arrow[scale=1.3]{<}}}}](0,0) +(-1,-1) -- +(1,1)
      node[below,at start] {$i$}; \draw[very thick,postaction={decorate,decoration={markings,
    mark=at position .2 with {\arrow[scale=1.3]{<}}}}](0,0) +(1,-1) --
      +(-1,1) node[below,at start] {$i$}; \fill (.5,.5) circle (3pt);
      \node at (2,0){$=$}; \draw[very thick,postaction={decorate,decoration={markings,
    mark=at position .5 with {\arrow[scale=1.3]{<}}}}](4,0) +(-1,-1) -- +(-1,1)
      node[below,at start] {$i$}; \draw[very thick,postaction={decorate,decoration={markings,
    mark=at position .5 with {\arrow[scale=1.3]{<}}}}](4,0) +(0,-1) --
      +(0,1) node[below,at start] {$i$};
    \end{tikzpicture}
  \end{equation*}
  \begin{equation*}\subeqn\label{black-bigon}
    \begin{tikzpicture}[very thick,scale=.9,baseline]
      \draw[postaction={decorate,decoration={markings,
    mark=at position .5 with {\arrow[scale=1.3]{<}}}}] (-2.8,0) +(0,-1) .. controls (-1.2,0) ..  +(0,1)
      node[below,at start]{$i$}; \draw[postaction={decorate,decoration={markings,
    mark=at position .5 with {\arrow[scale=1.3]{<}}}}] (-1.2,0) +(0,-1) .. controls
      (-2.8,0) ..  +(0,1) node[below,at start]{$i$}; \node at (-.5,0)
      {=}; \node at (0.4,0) {$0$};
\node at (1.5,.05) {and};
    \end{tikzpicture}
\hspace{.4cm}
    \begin{tikzpicture}[very thick,scale=.9,baseline]

      \draw[postaction={decorate,decoration={markings,
    mark=at position .5 with {\arrow[scale=1.3]{<}}}}] (-2.8,0) +(0,-1) .. controls (-1.2,0) ..  +(0,1)
      node[below,at start]{$i$}; \draw[postaction={decorate,decoration={markings,
    mark=at position .5 with {\arrow[scale=1.3]{<}}}}] (-1.2,0) +(0,-1) .. controls
      (-2.8,0) ..  +(0,1) node[below,at start]{$j$}; \node at (-.5,0)
      {=}; 
\draw (1.8,0) +(0,-1) -- +(0,1) node[below,at start]{$j$};
      \draw (1,0) +(0,-1) -- +(0,1) node[below,at start]{$i$}; 
\node[inner xsep=10pt,fill=white,draw,inner ysep=8pt] at (1.4,0) {$Q_{ij}(y_1,y_2)$};
    \end{tikzpicture}
  \end{equation*}
 \begin{equation*}\subeqn\label{triple-dumb}
    \begin{tikzpicture}[very thick,scale=.9,baseline]
      \draw[postaction={decorate,decoration={markings,
    mark=at position .2 with {\arrow[scale=1.3]{<}}}}] (-3,0) +(1,-1) -- +(-1,1) node[below,at start]{$k$}; \draw[postaction={decorate,decoration={markings,
    mark=at position .8 with {\arrow[scale=1.3]{<}}}}]
      (-3,0) +(-1,-1) -- +(1,1) node[below,at start]{$i$}; \draw[postaction={decorate,decoration={markings,
    mark=at position .5 with {\arrow[scale=1.3]{<}}}}]
      (-3,0) +(0,-1) .. controls (-4,0) ..  +(0,1) node[below,at
      start]{$j$}; \node at (-1,0) {=}; \draw[postaction={decorate,decoration={markings,
    mark=at position .8 with {\arrow[scale=1.3]{<}}}}] (1,0) +(1,-1) -- +(-1,1)
      node[below,at start]{$k$}; \draw[postaction={decorate,decoration={markings,
    mark=at position .2 with {\arrow[scale=1.3]{<}}}}] (1,0) +(-1,-1) -- +(1,1)
      node[below,at start]{$i$}; \draw[postaction={decorate,decoration={markings,
    mark=at position .5 with {\arrow[scale=1.3]{<}}}}] (1,0) +(0,-1) .. controls
      (2,0) ..  +(0,1) node[below,at start]{$j$}; \node at (5,0)
      {unless $i=k\neq j$};
    \end{tikzpicture}
  \end{equation*}
\begin{equation*}\subeqn\label{triple-smart}
    \begin{tikzpicture}[very thick,scale=.9,baseline]
      \draw[postaction={decorate,decoration={markings,
    mark=at position .2 with {\arrow[scale=1.3]{<}}}}] (-3,0) +(1,-1) -- +(-1,1) node[below,at start]{$i$}; \draw[postaction={decorate,decoration={markings,
    mark=at position .8 with {\arrow[scale=1.3]{<}}}}]
      (-3,0) +(-1,-1) -- +(1,1) node[below,at start]{$i$}; \draw[postaction={decorate,decoration={markings,
    mark=at position .5 with {\arrow[scale=1.3]{<}}}}]
      (-3,0) +(0,-1) .. controls (-4,0) ..  +(0,1) node[below,at
      start]{$j$}; \node at (-1,0) {=}; \draw[postaction={decorate,decoration={markings,
    mark=at position .8 with {\arrow[scale=1.3]{<}}}}] (1,0) +(1,-1) -- +(-1,1)
      node[below,at start]{$i$}; \draw[postaction={decorate,decoration={markings,
    mark=at position .2 with {\arrow[scale=1.3]{<}}}}] (1,0) +(-1,-1) -- +(1,1)
      node[below,at start]{$i$}; \draw[postaction={decorate,decoration={markings,
    mark=at position .5 with {\arrow[scale=1.3]{<}}}}] (1,0) +(0,-1) .. controls
      (2,0) ..  +(0,1) node[below,at start]{$j$}; \node at (2.8,0)
      {$+$};        \draw (6.2,0)
      +(1,-1) -- +(1,1) node[below,at start]{$i$}; \draw (6.2,0)
      +(-1,-1) -- +(-1,1) node[below,at start]{$i$}; \draw (6.2,0)
      +(0,-1) -- +(0,1) node[below,at start]{$j$}; 
\node[inner ysep=8pt,inner xsep=5pt,fill=white,draw,scale=.8] at (6.2,0){$\displaystyle \frac{Q_{ij}(y_3,y_2)-Q_{ij}(y_1,y_2)}{y_3-y_1}$};
    \end{tikzpicture}
  \end{equation*}
\item the composition \[
\tikz[very thick]{
\draw[postaction={decorate,decoration={markings,
    mark=at position .8 with {\arrow[scale=1.3]{>}}}}] (0,0) to[out=90,in=-90] node[below,at start]{$j$} node[above,at end]{$j$} (1,1);\draw[postaction={decorate,decoration={markings,
    mark=at position .8 with {\arrow[scale=1.3]{<}}}},postaction={decorate,decoration={markings,
    mark=at position .2 with {\arrow[scale=1.3]{<}}}}] (-.5,0)
  to[out=90,  in=180] node[below,at start]{$i$} (0,.7) to[out=0, in=180] (1,.3) to[out=0,in=-90] node[above,at end]{$i$} (1.5,1); }
\] possesses an inverse.
\item if $\la^i\geq 0$, then the map   $\sigma_{\la,i}\colon \eE_i\eF_i\to \eF_i \eE_i\oplus
  \id_\la^{\oplus\la^i}$ given by  
\[
\tikz[very thick,scale=1.4]{\node[scale=1.5] at (-1,.5){$\la$};
\draw[postaction={decorate,decoration={markings,
    mark=at position .8 with {\arrow[scale=1.3]{>}}}}] (0,0) to[out=90,in=-90] node[below,at start]{$i$} node[above,at end]{$i$} (1,1);\draw[postaction={decorate,decoration={markings,
    mark=at position .8 with {\arrow[scale=1.3]{<}}}},postaction={decorate,decoration={markings,
    mark=at position .2 with {\arrow[scale=1.3]{<}}}}] (-.5,0)
  to[out=90,  in=180] node[below,at start]{$i$} (0,.7) to[out=0,
  in=180] (1,.3) to[out=0,in=-90] node[above,at end]{$i$} (1.5,1); 
\node at (2,.5){$\oplus$};
\draw[postaction={decorate,decoration={markings,
    mark=at position .8 with {\arrow[scale=1.3]{<}}}},postaction={decorate,decoration={markings,
    mark=at position .3 with {\arrow[scale=1.3]{<}}}}] (2.5,0)
  to[out=90,  in=180] node[below,at start]{$i$} (2.75,.5)
  to[out=0,in=90] node[below,at end]{$i$} (3,0); 
\node at (3.5,.5){$\oplus$};
\draw[postaction={decorate,decoration={markings,
    mark=at position .8 with {\arrow[scale=1.3]{<}}}},postaction={decorate,decoration={markings,
    mark=at position .3 with {\arrow[scale=1.3]{<}}}}] (4,0)
  to[out=90,  in=180] node[below,at start]{$i$} node[at
  end,circle,fill=black,inner sep=2pt]{} (4.25,.5)
  to[out=0,in=90] node[below,at end]{$i$} (4.5,0); 
\node at (5,.5){$\oplus$};
\node at (5.5,.5){$\dots$};
\node at (6,.5){$\oplus$};
\draw[postaction={decorate,decoration={markings,
    mark=at position .8 with {\arrow[scale=1.3]{<}}}},postaction={decorate,decoration={markings,
    mark=at position .3 with {\arrow[scale=1.3]{<}}}}] (6.5,0)
  to[out=90,  in=180] node[below,at start]{$i$} node[at
  end,circle,fill=black,inner sep=2pt]{} (6.75,.5)
  to[out=0,in=90] node[below,at end]{$i$} (7,0);\node at (6.75,.75){$\lambda^i-1$};  }
\] possesses an inverse.
\item if $\la^i\leq 0$, then the map   $\sigma_{\la,i}\colon \eE_i\eF_i\oplus
  \id_\la^{\oplus-\la^i}\to \eF_i \eE_i$ given by  
\[
\tikz[very thick,scale=1.4]{\node[scale=1.5] at (-1,.5){$\la$};
\draw[postaction={decorate,decoration={markings,
    mark=at position .8 with {\arrow[scale=1.3]{>}}}}] (0,0) to[out=90,in=-90] node[below,at start]{$i$} node[above,at end]{$i$} (1,1);\draw[postaction={decorate,decoration={markings,
    mark=at position .8 with {\arrow[scale=1.3]{<}}}},postaction={decorate,decoration={markings,
    mark=at position .2 with {\arrow[scale=1.3]{<}}}}] (-.5,0)
  to[out=90,  in=180] node[below,at start]{$i$} (0,.7) to[out=0,
  in=180] (1,.3) to[out=0,in=-90] node[above,at end]{$i$} (1.5,1); 
\node at (2,.5){$\oplus$};
\draw[postaction={decorate,decoration={markings,
    mark=at position .8 with {\arrow[scale=1.3]{<}}}},postaction={decorate,decoration={markings,
    mark=at position .3 with {\arrow[scale=1.3]{<}}}}] (2.5,1)
  to[out=-90,  in=180] node[above,at start]{$i$} (2.75,.5)
  to[out=0,in=-90] node[above,at end]{$i$} (3,1); 
\node at (3.5,.5){$\oplus$};
\draw[postaction={decorate,decoration={markings,
    mark=at position .8 with {\arrow[scale=1.3]{<}}}},postaction={decorate,decoration={markings,
    mark=at position .3 with {\arrow[scale=1.3]{<}}}}] (4,1)
  to[out=-90,  in=180] node[above,at start]{$i$} node[at
  end,circle,fill=black,inner sep=2pt]{} (4.25,.5)
  to[out=0,in=-90] node[above,at end]{$i$} (4.5,1); 
\node at (5,.5){$\oplus$};
\node at (5.5,.5){$\dots$};
\node at (6,.5){$\oplus$};
\draw[postaction={decorate,decoration={markings,
    mark=at position .8 with {\arrow[scale=1.3]{<}}}},postaction={decorate,decoration={markings,
    mark=at position .3 with {\arrow[scale=1.3]{<}}}}] (6.5,1)
  to[out=-90,  in=180] node[above,at start]{$i$} node[at
  end,circle,fill=black,inner sep=2pt]{} (6.75,.5)
  to[out=0,in=-90] node[above,at end]{$i$} (7,1); \node at (6.75,.25){$-\lambda^i-1$}; }
\] possesses an inverse.
\end{itemize}
\end{definition}

\subsection{Dualities}
\label{sec:pivotal-structures}

In this category, the functors $\eE_i$ and $\eF_i$ are biadjoint up to
shift.  Some care about this biadjunction is needed, since it is not
unique up to isomorphism.  However, since the degree 0 automorphisms
of $\eF_i$ are simply a copy of the scalars, the biadjunction is unique up to
scalar multiplication.  

Choosing such a biadjunction between $\eE_i1_{\la}$ and $\eF_i
1_{\la+\al_i}$ for each $i$ and $\la$ defines a duality functor on
$\tU$ such that 
\[(\eE_i1_{\la})^\star=\eF_i
1_{\la+\al_i} (\langle \la,\al_i\rangle +1)\qquad (\eF_i1_{\la+\al_i})^\star=\eE_i
1_{\la} (-\langle \la,\al_i\rangle -1)\] and $u^\star$ is right adjoint
to $u$. 
In this case, $u= u^{\star\star}$ up to shift for every 1-morphism $u$.

We can define one such duality by defining 
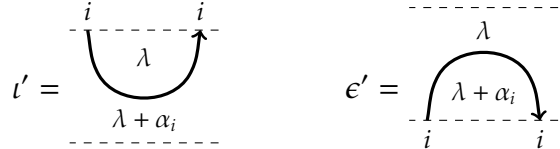
\[\iota'=\tikz[baseline,very thick,scale=3]{\draw[<-]   (.25,.3) to
  [out=-100, in=60] node[at start,above,scale=.8]{$i$} (.2,.1)
  to[out=-120,in=-60] (-.2,.1) to [out=120,in=-80] 
  node[at end,above,scale=.8]{$i$} (-.25,.3);\node[scale=.8] at
  (0,.18){$\la$}; \node[scale=.8] at (0,-.1){$\la+\al_i$};
  \draw[thin,dashed] (.33,.3) -- (-.33,.3);\draw[thin,dashed] (.33,-.2) -- (-.33,-.2);}\qquad\qquad 
\ep'=\tikz[baseline,very thick,scale=3]{\draw[<-] (.25,-.1) to
  [out=100, in=-60]   node[at start,below,scale=.8]{$i$} (.2,.1)
  to[out=120,in=60]
  (-.2,.1)  to [out=-120,in=80] node[at end,below,scale=.8]{$i$} (-.25,-.1);\node[scale=.8] at
  (0,.3){$\la$};\node[scale=.8] at (0,.02){$\la+\al_i$};
  \draw[thin,dashed] (.33,-.1) -- (-.33,-.1); \draw[thin,dashed]
  (.33,.4) -- (-.33,.4);}\] according to the rule of
\cite[(1.14-18)]{Brundandef}.  We call this the {\bf Cautis-Lauda
 duality}, since it is uniquely characterized by the relations in
\cite[\S 2]{CaLa}.  However, this choice of duality is not cyclic; that
is, the double dual of a morphism might not coincide with the original
morphism.  To define an action on  $Z(R^\la_\mu)$, we need a cyclic
duality.  Such a duality is constructed in \cite[\S 9]{BHLW}, and will be
discussed in greater detail in \cite{BHLWcyc}.  For our purposes, it will
be useful to give a self-contained account here.

To start with, we can define other dualities by how they differ from the
Cautis-Lauda duality.
\begin{definition}
  Given a map $\theta\colon \wela\to \K^{\times}$, and a biadjunction
  $(\eE_i,\eF_i,\epsilon,\iota,\epsilon',\iota')$, we let the twist of
  this biadjunction be the biadjunction which leaves  $\epsilon,\iota$
  unchanged and takes $\tilde{\epsilon}'_i\colon \eF_i\eE_i1_\la\to
  1_\la$ to be $\theta(\la) \epsilon'$; consequently, we must take $\tilde{\iota}'_i
\colon 1_\la\to
  \eE_i\eF_i1_\la$ to be $\theta(\la+\al_i)^{-1}\iota_i'$. 
\end{definition}

It will be useful for us to sometimes use partially defined maps,
since for an indecomposable categorical module, there will only be
non-zero categories associated to a single coset of the root lattice $\rola$ in
the weight lattice $\wela$.  

Now, let $\beta\colon \rola\times \rola\to \Z/2\Z$ be a bilinear
map of abelian groups such that we
have \[\beta(\gamma,\delta)+\beta(\delta,\gamma)\equiv \langle
\gamma,\delta\rangle\pmod 2.\]

Such $\beta$ obviously form an affine space over the abelian group of
symmetric $\Z/2\Z$-valued forms on $\rola$.  
Note that we can very easily show that this space is non-empty by picking an
order on roots, and declaring that 
\[\beta(\al_i,\al_j)=
\begin{cases}
  \langle
\al_i,\al_j\rangle & i>j\\
0 & j\leq i.
\end{cases}\]

In \cite{Varyang}, such a function is produced by a pair of maps  $(-)_\pm\colon\rola\to \rola$ such that $\la=\la_++\la_-$ and
$\langle \la_+,\mu\rangle=\langle \la,\mu_-\rangle$, and setting $\beta(\gamma,\delta)=\langle
\gamma_+,\delta\rangle\pmod 2$. 
In \cite{BHLW}, we work by assuming that $\Gamma=\Gamma_0\cup \Gamma_1$ is bipartite, and take \[\beta(\al_i,\al_j)=
\begin{cases}
  \langle \al_i,\al_j\rangle & i\in\Gamma_0\\
  0  & i\in \Gamma_1
\end{cases}\]

Pick a map
$\Pi\colon \wela\to \rola$ compatible with the
$\rola$ action by addition.
Let the $\beta$-Cautis-Lauda duality be that obtained by
twisting the Cautis-Lauda duality by the function
$\mu\mapsto (-1)^{\beta(\Pi(\mu),\al_i)}$.  

\begin{proposition}
For any such $\beta$, the 
$\beta$-Cautis-Lauda
duality is cyclic.
\end{proposition}
\begin{proof}
  The Cautis-Lauda duality fails to be cyclic because the
  crossing $\psi$ of differently colored strands may not be.  In fact,
  in our conventions, a full counter clockwise rotation of this
  diagram differs from the original by a factor of
  $v_{ij}=t_{ij}/t_{ji}=( -1)^{\langle \al_i,\al_j\rangle}$.  In \cite[2.5]{CaLa},
  the left-hand side of the equation changes by $-1$ raised to
  \begin{align*}
\beta( \la,\al_j)+\beta( \la+\al_j,\al_i)-2\beta( \la,\al_i)
    -\beta( \la+\al_i,\al_j)
&=\beta( \al_i,\al_j)-2\beta (\al_j,\al_i)\\
&=\langle \al_i,\al_j\rangle.
  \end{align*}
  Thus, $\psi$ is cyclic in our geometric biadjoint.  It's clear that
  the cyclicity of $y$ remains unchanged, so the result is proved.
\end{proof}

\begin{proposition}[\mbox{\cite[5.1]{BHLW}} ]\label{prop-cur}
The elements  of $\Tr(\tU)$ defined by:
\[\mathsf{E}_{i,r}1_{\lambda}:=  \left[
\tikz[baseline=-2pt,very thick,scale=2.5]{\draw[<-]
  (0,.3) -- (0,-.3) 
node[at end,below,scale=.8]{$i$} node[at start,above,scale=.8]{$i$}
  node[midway,circle,fill=black,inner
  sep=2pt,label=right:{$r$}]{}; 
\node[scale=.8] at (-.2,0) {$\la$}; 
}
\right],  \quad \quad
\mathsf{F}_{j,s}1_{\lambda}:=  \left[
\tikz[baseline=-2pt,very thick,scale=2.5]{\draw[->]
  (0,.3) -- (0,-.3) 
node[at end,below,scale=.8]{$i$} node[at start,above,scale=.8]{$i$}
  node[midway,circle,fill=black,inner
  sep=2pt,label=right:{$r$}]{}; 
\node[scale=.8] at (-.2,0) {$\la$}; 
}
\right],\quad \quad
  \mathsf{H}_{i,r}1_{\lambda}:= \left[
 p_{i,r}(\lambda)\colon \id_{\lambda}\to \id_{\lambda}
\right],\]
define a homomorphism
\begin{equation} \label{eq_sln-homo}
 \rho \colon \dot{\bf U}(\fg[t]) \longrightarrow \Tr(\tU),
\end{equation}
given by
\begin{equation} \label{homomorp}
x^{+}_{i, r} 1_{\lambda} \mapsto
\mathsf{E}_{i,r} 1_{\lambda}
 ,\quad \quad
x^{-}_{j, s} 1_{\lambda}\mapsto
\mathsf{F}_{j,s} 1_{\lambda}
,\quad \quad
\xi_{i, r} 1_{\lambda}\mapsto
\mathsf{H}_{i,r}1_{\lambda}.
\end{equation}
\end{proposition}

As explained in \cite[\S 9]{BHLW}, this map induces an action of the current
algebra on the 
cocenter of any category $\mathcal{C}$ with a categorical
$\fg$-action.  Furthermore, if we choose a cyclic duality, then we also
have an action on the center by wrapping with bubbles, as in the graphical calculus
of \cite{CW}.  More precisely, if we have a 
1-morphism $u\colon \mu \to \nu$, and an endomorphism $a\colon u\to
u$, then we can define the convolution
$a\star - \colon Z(\mathcal{C}_\mu)\to Z(\mathcal{C}_\nu)$ given by
\[a\star z=\ep_{u}\circ (a\otimes z\otimes 1_{u^\star})\circ
\iota_{u^\star}\]

It will sometimes be useful to factor this convolution into the
natural map $\End(u,u)\otimes Z(\mathcal{C}_\mu)\to \End(u,u)$ and
then a trace map $\tau\colon \End(u,u)\to Z(\mathcal{C}_\nu)$ given by
$\tau(a)=a\star 1_\mu$.  In this notation, the action of the current
algebra is given by $x^\pm_{i,r}(z)=y_{\pm i}^r\star z$ where $y_{\pm
  i}$ is the dot endomorphism of $\eE_{\pm i}$.  For simplicity, we
let $\tau^\star(a)=\tau(a^\star)$.  
  
\subsection{Cyclotomic quotients}
\label{sec:cyclotomic-quotients}

One of the actions of this 2-category that interests us is on  the
modules over 
cyclotomic KLR algebras $R^\al_\mu$  and their natural deformations
$\check{R}^\la_\mu$ defined by Khovanov and Lauda \cite{KLI} and
Rouquier \cite{Rou2KM,RouQH}.  These play a
universal role amongst all categorifications of the simple
representations with a given highest weight.  We refer to
\cite[\S\ref{m-sec:cyc}\,\&\,\ref{m-sec:univ-quant}]{Webmerged} for their definition, and we will follow the notation of
that paper.  As mentioned above, the center and cocenter of these
algebras inherit a current algebra action, analyzed extensively in
\cite{SVVcenter}.  We'll say that {\bf algebraic Kirwan surjectivity} holds for
a given quiver if the map $\kappa_a\colon Z(R_{\la-\mu})\to
Z(R^\la_\mu)$ for all $\mu$.   This follows immediately
if the sum $\oplus_\mu Z(R^\la_\mu)$ is generated as a
module over $\fg[t]$ by the identities $1_{\mu}$ in $R^\la_\mu$, since
the image of $\kappa_a$ is closed under the current algebra action.

\begin{proposition}\label{prop:Kirwan}
Algebraic  Kirwan surjectivity holds for finite type ADE.
\end{proposition}
\begin{proof}
 This is essentially equivalent to \cite[7.3]{BHLW} or \cite[Th. 1]{SVVcenter}.  By
  \cite{Naoi}, the dual Weyl module has a unique simple quotient which
  is a copy of $W_{\la_{min}}$ where $\la_{min}$ is a dominant weight
  of minimal norm such that $\la_{min}\leq \la$.  Furthermore, this
  copy has degree 0, and the quotient map kills all elements of
  positive degree in $Z(R^{\la}_{\la_{min}})$.  Thus, the identity of
  this algebra has non-zero image in this simple quotient.  This shows
  that it is a generator, since it is not contained in the unique
  maximal submodule.  
\end{proof}

\section{The quantum geometry of quiver varieties}
\label{sec:quant-geom-quiv}

The other categorical representation that will interest us is 
a geometric one arising in \cite{Webcatq}.  This representation
$\mathcal{G}_\la$ depends
on a choice of a highest weight $\la$ used in the definition of the
underlying quiver varieties.   This representation is a direct
quantization/categorification of Nakajima's construction in the sense
that $\eE_i, \eF_i$ are sent to functors of convolutions with sheaves
$\mathscr{E}_i,\mathscr{F}_i$ defined in \cite[(\ref{q-Fdef}--\ref{q-Edef})]{Webcatq}.  The sheaf $ \mathscr{E}_i$
is a quantization of the structure sheaf on the Hecke
correspondence $\mathfrak{P}_i$ used by Nakajima, which is a smooth Lagrangian
subvariety of 
$\fM^\la_{\mu+\al_i}\times\fM^\la_{\mu}$. 

Assume that $i$ is a source of $\Gamma$.  The functors $\mathscr{E}_i,\mathscr{F}_i$ are constructed as
Hamiltonian reductions of the convolution on D-modules along the
correspondence
\[\hat{X}^\la_{\mu-\al_i}\longleftarrow\hat{X}^\la_{\mu;\al_i}\longrightarrow\hat{X}^\la_{\mu}.\]
where $\hat{X}^\la_{\mu}$ is the space of framed quiver representations on a
fixed underlying vector space, modulo $GL(\C^{v_i})$, as defined in \cite[\ref{q-def:hat}]{Webcatq}.  One can view
the points of 
this space as a quiver representation on $\Gamma\setminus\{i\}$ and
choice of subspace $V_i$ of dimension $v_i$ in $V_{out}\cong \sum_{i\to
  j}V_j$.   The space 
$\hat{X}^\la_{\mu;\al_i}$ is the analogous space with a nested pair
$V_i\subset V_i'$ with dimensions $v_i$ and $v_i+1$.  The projections
forgetting either of these spaces is proper and smooth with fibers
given by projective spaces.  Thus, Poincar\'e duality on the fibers
with the complex orientation induces a biadjunction of
$(\mathscr{E}_i,\mathscr{F}_i)$, and so a duality on the
image of $\tU$ under $\mathcal{G}_\la$.  We call this the {\bf
  geometric duality}\footnote{Of course ``Poincar\'e duality'' or
  ``Verdier duality'' would both be very appropriate names, but could
  easily lead to confusion.}. As with the Cautis-Lauda duality, the
$\beta$-geometric duality is twist of the geometric duality by
$\mu\mapsto (-1)^{\beta(\mu-\la,\al_i)}$.   

 Every other weight that appears in this
categorification is of the form $\mu=\la-\sum v_i\al_i$.   We'll consider
the function on this coset that sends $\xi_i(\mu)=\mu^i-v_i$.  Note that
this is the same as the number of times $F_i$ must be applied to bring
$\mu$ to the boundary of the Weyl polytope.  In finite type, this is
the same as $m_i$ where $\mu=w_0\la+\sum_jm_j\al_j$.

\begin{proposition}
  The representation $\mathcal{G}_\la$ sends the Cautis-Lauda duality to the geometric
  duality twisted by the function $\mu\mapsto (-1)^{\xi_i(\mu)}$. 
\end{proposition}
\begin{proof}
  In order to calculate the twist, we need only calculate the value of
  one non-zero diagram using both dualities and compare
  these.  Choose a weight $\mu$; assume for now that $\mu^i\leq 0$,
  and let $m_i=\xi(\mu)=v_i+\mu^i$.
  In this case, in the Cautis-Lauda duality, we have that
  $\epsilon'_i(1\otimes y^{-\mu^i+1})\iota_i=1$ on $1_{\mu-\al_i}$.
  Thus, we must compute this value in the geometric duality and see
  that we obtain $(-1)^{m_i}$.    That is, we must compute the
  convolution $(q_1)_*q_2^*(c_1(V_i'/V_i)^{-\mu^i+1} \cup (q_2)_*q_1^*1)$ using the diagram
  \[\hat{X}^\la_{\mu-\al_i}\overset{q_1}\longleftarrow\hat{X}^\la_{\mu;\al_i}\overset{q_2}\longrightarrow\hat{X}^\la_{\mu;\al_i}\times_{\hat{X}^\la_{\mu}}
  \hat{X}^\la_{\mu;\al_i}.\] 

The normal bundle of $\hat{X}^\la_{\mu;\al_i}$ in $\hat{X}^\la_{\mu;\al_i}\times_{\hat{X}^\la_{\mu}}
  \hat{X}^\la_{\mu;\al_i}$ is given by tangent vectors giving variations of
  $V_i'$ with 
  $V_{out} \supset V_i'\supset V_i$; this space is isomorphic to $\Hom(V_i'/V_i,V_{out}/V_i')$;
  thus, letting $z:=c_1(V_i'/V_i)$, the resulting Euler class is
  \[e(V_i'/V_i,V_{out}/V_i')=(-z)^{m_i-1}+(-z)^{m_i-2}c_1(V_{out}/V_i')+\cdots
  +c_{v_i-1}(V_{out}/V_i').\]  Thus, we have that 
\[z^{-\mu^i+1} \cup (q_2)_*q_1^*1=(-1)^{m_i}(z^{v_i}-z^{v_i-1}c_1(V_{out}/V_i')+\cdots
  +(-1)^{m_i+1} z^{-\mu_i}c_{v_i-1}(V_{out}/V_i')).\]
The fibers of the map $q_1$ are isomorphic to projective spaces
$\mathbb{P}^{v_i}$; let $j$ be the inclusion of such a fiber.
The restriction of $V_{out}/V_i'$ to a fiber is trivial, so
$j^*c(V_{out}/V_i')=1$, while $V_i'/V_i$ pulls back to
$\mathcal{O}(1)$, so $j^*z=H$, the standard hyperplane class.  Thus
\[j^* (z^{-\mu^i} \cup (q_2)_*q_1^*1)=(-1)^{m_i}H^{v_i}.\]  The
integral of this over the fiber is, of course, $(-1)^{m_i}$, so we
have that \[(q_1)_*q_2^*(c_1(V_i'/V_i)^{-\mu^i} \cup
(q_2)_*q_1^*1)=(-1)^{m_i}\] as desired.  

The calculation when $m_i\geq 0$ is similar.  In this case, we use
that $\epsilon'_i(1\otimes y^{\mu^i+1})\iota_i=1$ on $1_{\mu+\al_i}$.   That is, we must compute the
  convolution $(r_1)_*r_2^*(z^{\mu^i+1} \cup (r_2)_*r_1^*1)$ using the diagram
  \[\hat{X}^\la_{\mu+\al_i}\overset{r_1}\longleftarrow\hat{X}^\la_{\mu+\al_i;\al_i}\overset{r_2}\longrightarrow\hat{X}^\la_{\mu+\al_i;\al_i}\times_{\hat{X}^\la_{\mu}}
  \hat{X}^\la_{\mu+\al_i;\al_i}.\]  The normal bundle to $
  \hat{X}^\la_{\mu+\al_i;\al_i}$ in $\hat{X}^\la_{\mu+\al_i;\al_i}\times_{\hat{X}^\la_{\mu}}
  \hat{X}^\la_{\mu+\al_i;\al_i}$ corresponds to variations of
  $V_i\subset V_i'$ and thus is given by $\Hom(V_i,V_i'/V_i)$ so we have that 
\[z^{\mu^i+1} \cup (r_2)_*r_1^*1=z^{m_i}-z^{m_i-1}c_1(V_{out}/V_i')+\cdots
  +(-1)^{m_i-1} z^{v_i+1}c_{\mu_i-1}(V_{out}/V_i')).\]

Now, the fibers of $r_1$ are projective spaces $\mathbb{P}^{m_i}$, and
$V_i'/V_i$ pulls back to the line bundle $\mathcal{O}(-1)$, so $z$
corresponds to $-H$, and the integral of $z^{m_i}$ is $(-1)^{m_i}$, as desired.
\end{proof}

While the Cautis-Lauda and geometric dualities are different, they induce the same double dual map:
\begin{proposition}\label{prop:double-dual}
  The double dual maps of the Cautis-Lauda and geometric
  biadjunctions agree.  In particular, the $\beta$-geometric duality is cyclic.
\end{proposition}
\begin{proof}
  For the double dual of $y$, this is clear.  For $\psi$, the
  consider the diagram \cite[2.5]{CaLa}.  The 2 cups and 2 caps of
  this diagram contribute a factor of $-1$ raised to 
\[\xi_j(\mu)+\xi_i(\mu-\al_j)-\xi_i(\mu)-\xi_j(\mu-\al_i)=-v_j-\langle
\al_i,\al_j\rangle -v_i+v_i+\langle
\al_j,\al_i\rangle+v_j=0.\]  This completes the proof.
\end{proof}

Since this representation sends grading shift to shift in the
dg-category of $\A_\mu$-modules, it induces a current algebra action on the
center of this category in the dg-sense: its Hochschild cohomolgy
$H\!H(\A_\mu)$.  

As discussed before and in \cite[\S 5.4]{BLPWgco}, we have a natural map
$H\!H(\A_\mu)\to Z(\Ext^\bullet(M,M))$ for any sheaf of $\A_\mu$-modules
$M$.  By \cite[3.6]{Webcatq}, there is a semi-simple core module
$C_\mu$ such that $\Ext^\bullet(C_\mu,C_\mu)\cong R^\la_\mu$, with
this isomorphism inducing an equivalence of dg-categories between the
categories $\mathcal{C}^\la_\mu$ of core modules and
$R^\la_\mu\dgmod$ compatible with the categorical actions of
$\mathcal{C}^\la_\mu$ and $R^\la_\mu\dgmod$.  
That is:
\begin{corollary}\label{cor:centers-match}
  For any cyclic duality, the map $ H^*(\fM^\la_\mu)\cong H\!H(\A_\mu)\to
  Z(\Ext^\bullet(C_\mu,C_\mu))\cong Z(R^\la_\mu)$ is compatible with
  the current algebra action induced by that duality on both sides.
\end{corollary}
This is half of our main theorem, since it relates the current algebra
action of $Z(R^\la_\mu)$ to the cohomology of $ H^*(\fM^\la_\mu)$.
However, it is not obvious that the action of $H^*(\fM^\la_\mu)$
induced by
the $\tU$-action agrees with Varagnolo's. We turn to showing this in
the next section.

\section{Hochschild cohomology and Betti cohomology}
\label{sec:hochsch-cohom}

\subsection{Hochschild vs. Betti} 
Now let us discuss the connection of the categorical action
$\mathcal{G}_\la$ to the usual pushforward and pullback operations on
cohomology.  As before, let $\mathfrak{P}_i \subset
\fM^\la_{\mu+\al_i}\times \fM^\la_{\mu}$ be the Hecke
correspondence for $i$; let $U$ be the open complement of this variety.  This subvariety has projection maps
\[\pi_1\colon   \mathfrak{P}_i\to \fM^\la_{\mu+\al_i}\qquad
\pi_2\colon \mathfrak{P}_i\to \fM^\la_{\mu}.\]  For simplicity, let $d_{\mu}=\dim_\C \fM^\la_{\mu}$.  There is the usual
pullback map in cohomology $\pi_i^*$, but we'll also wish to consider
the pushforwards
\[(\pi_1)_*\colon H^*(\mathfrak{P}_i )\cong
H_{d_{\mu}+d_{\mu+\al_i}-*}^{BM}(\mathfrak{P}_i )\to
H_{d_{\mu}+d_{\mu+\al_i}-*}^{BM} (\fM^\la_{\mu+\al_i})\cong
H^{-d_{\mu}+d_{\mu+\al_i}+*}(\fM^\la_{\mu+\al_i}).\]
\[(\pi_2)_*\colon H^*(\mathfrak{P}_i )\cong
H_{d_{\mu}+d_{\mu+\al_i}-*}^{BM}(\mathfrak{P}_i )\to
H_{d_{\mu}+d_{\mu+\al_i}-*}^{BM} (\fM^\la_{\mu})\cong
H^{d_{\mu}-d_{\mu+\al_i}+*}(\fM^\la_{\mu}).\]

By definition, Varagnolo's current algebra operators \cite[(4.1)]{Varyang} depend on a
choice of $\beta$, and are given by:
\[x^+_{i,r}(m)=(-1)^{\beta(\mu-\la,\al_i)}(\pi_1)_*(z^r\cup
\pi_2^*m)\qquad x^-_{i,r}(m)=(-1)^{\beta(\al_i,\mu-\la)}(\pi_2)_*(z^r\cup
\pi_1^*m)\] where as above, $z=c_1(V_i'/V_i)$ is the Chern class of
the tautological line bundle over the Hecke correspondence.  Note that
Varagnolo uses a more restricted set of $\beta$'s, but one can easily check
that changing this function by a symmetric bilinear form only changes the action by an automorphism
of the Yangian.

\begin{lemma}
  We have an isomorphism $\Ext^\bullet(\mathscr{E}_i,
  \mathscr{E}_i)\cong H^*(\mathfrak{P}_i;\C)$ and the natural algebra maps
\[H\!H(\A_{\mu+\al_i})\to  \mathscr{E}xt^\bullet(\mathscr{E}_i,
  \mathscr{E}_i)\leftarrow H\!H(\A_{\mu})\] are intertwined with the
  pullbacks $\pi_1^*,\pi_2^*$.
\end{lemma}
\begin{proof}
  The proof closely follows the computation of the Hochschild
  cohomology of $\A$ in \cite[\S
5.4]{BLPWgco}.  Since $\mathfrak{P}_i$ is smooth, computing the sheaf
Ext $\mathscr{E}xt^\bullet(\mathscr{E}_i,
  \mathscr{E}_i)$ reduces to a local computation with the vacuum
  representation of the Weyl algebra.  This has vanishing higher
  self-Exts and 1 dimensional in degree 0, as the usual Koszul
  resolution shows.  Thus, $ \mathscr{E}xt^\bullet(\mathscr{E}_i,
  \mathscr{E}_i)\cong \C_{\mathfrak{P}_i}$, and so $\Ext^\bullet(\mathscr{E}_i,
  \mathscr{E}_i) \cong H^*(\mathfrak{P}_i;\C)$.  
This local computation also shows the match with pullback in cohomology, since they must be given by the
unique map of algebra sheaves between
$\pi_1^*\C_{\fM^\la_{\mu+\al_i}}$ (resp. $\pi_2^*\C_{\fM^\la_{\mu}}$) and $ \C_{\mathfrak{P}_i}$.
\end{proof}

If we fix a duality as before, then we also have 
pushforward maps \[H\!H(\A_{\mu+\al_i})\overset{\tau}\leftarrow\Ext^\bullet(\mathscr{E}_i,
  \mathscr{E}_i)\overset{\tau^\star}\to H\!H(\A_{\mu}),\] as defined
  in Section \ref{sec:pivotal-structures}.

  \begin{lemma}
 Under the geometric duality, the maps $\tau$ (resp. $\tau^\star$) are intertwined with the
    pushforward maps in cohomology $(\pi_1)_*$ (resp. $(\pi_2)_*$).  
  \end{lemma}
  \begin{proof}
    First, note that both maps are induced by sheaf maps from the
    pushforward  $(\pi_1)_* \C_{\mathfrak{P}_i }$ (resp. $(\pi_2)_* \C_{\mathfrak{P}_i }$) of the constant sheaf on $\mathfrak{P}_i$ to the
    constant sheaf $\C_{\fM^\la_{\mu+\al_i}}$
    (resp. $\C_{\fM^\la_{\mu}}$).

Since we have fixed $i$, we can choose an orientation for which $i$ is
a source.    In this case, $\fM^\la_\mu$ is a open subset of the stack
quotient $T^*\hat{X}^\la_\mu$, and   $\mathfrak{P}_i$ the
corresponding open subset of the conormal 
$N^*\hat{X}^\la_{\mu;\al_i}$.  The functors $\mathscr{E}_i$ and
$\mathscr{F}_i$ are induced by pushforward and pullback in the
category of D-modules along the maps
$\hat{X}^\la_{\mu}\leftarrow \hat{X}^\la_{\mu;\al_i}\to
\hat{X}^\la_{\mu+\al_i}$.   The geometric duality precisely matches
under 
Riemann-Hilbert correspondence with the usual
biadjunction for the corresponding functors on constructible sheaves.
That is, it corresponds to the natural sheaf maps  
\begin{equation}
(\pi_1)_* \C_{\hat{X}^\la_{\mu;\al_i}}\to
\C_{\hat{X}^\la_{\mu+\al_i}}[d_{\mu+\al_i}-d_{\mu}] \qquad (\pi_2)_*\C_{\hat{X}^\la_{\mu;\al_i}}\to
\C_{\hat{X}^\la_{\mu}}[d_{\mu}-d_{\mu+\al_i}]. \label{eq:2}
\end{equation}
On the other hand, Varagnolo's construction proceeds using similar
maps  \[(\pi_1)_* \C_{\mathfrak{P}_i}\to
\C_{\fM^\la_{\mu+\al_i}}[d_{\mu+\al_i}-d_{\mu}] \qquad (\pi_2)_*\C_{\mathfrak{P}_i}\to
\C_{\fM^\la_\mu}[d_{\mu}-d_{\mu+\al_i}], \] which are in turn the restriction to the stable
locus of sheaf maps
\begin{equation}
(\pi_1)_* \C_{N^*\hat{X}^\la_{\mu;\al_i}}\to
\C_{T^*\hat{X}^\la_{\mu+\al_i}}[d_{\mu+\al_i}-d_{\mu}] \qquad (\pi_2)_*\C_{N^*\hat{X}^\la_{\mu;\al_i}}\to
\C_{T^*\hat{X}^\la_{\mu}}[d_{\mu}-d_{\mu+\al_i}].\label{eq:1}
\end{equation}

While these varieties are different, we have vector bundle map 
\[q_1\colon  T^*\hat{X}^\la_{\mu+\al_i}\to
\hat{X}^\la_{\mu+\al_i}\qquad q_2\colon  T^*\hat{X}^\la_{\mu+\al_i}\to
\hat{X}^\la_{\mu+\al_i} \qquad q\colon N^*\hat{X}^\la_{\mu;\al_i}\to
\hat{X}^\la_{\mu;\al_i}\]
such that  \[(q_1)_*\C_{T^*\hat{X}^\la_{\mu+\al_i}}\cong
\C_{\hat{X}^\la_{\mu+\al_i}}\qquad (q_2)_*\C_{T^*\hat{X}^\la_{\mu}}\cong
\C_{\hat{X}^\la_{\mu}}\qquad q_*\C_{N^*\hat{X}^\la_{\mu;\al_i}}\cong
\C_{\hat{X}^\la_{\mu;\al_i}}\] and the pushforward map intertwines the
maps of \eqref{eq:2} and of \eqref{eq:1} by the usual functoriality.

 Thus, the map of sheaves induced by
$\hat{\mathscr{E}}_i$ and $\hat{\mathscr{F}}_i$  agrees the
pushforward and pullback of constructible sheaves on these larger varieties.  Since this is a
local property, it is unchanged by pulling back to the open subsets
$\fM^\la_\mu, \mathfrak{P}_i,$ and  $\fM^\la_{\mu+\al_i}$.
\end{proof}

Thus, immediately from the definition, we have that:
\begin{corollary}\label{cor:cohomology-match}
The isomorphism $H^*(\fM^\la_\mu)\to H\!H(\A_{\mu})$  intertwines
the action of the current algebra for any function $\beta$, and the action
induced by $\tU$ with the $\beta$-geometric duality.
\end{corollary}
This allows to complete the proof of the first part of our main theorem:
\begin{theorem}\label{thm:final}
  The induced map $\phi \colon H^*(\fM^\la_\mu)\to Z(R^\la_\mu)$ is injective.
\end{theorem}
\begin{proof}
Choose $\beta'$ so that the $\beta$-geometric and
$\beta'$-Cautis-Lauda dualities agree.  Combining Corollaries
\ref{cor:centers-match} and \ref{cor:cohomology-match}  show that $\phi$ intertwines the
$\beta$-action on $H^*(\fM^\la_\mu)$ with the $\beta'$-Cautis-Lauda
action on $Z(R^\la_\mu)$.  

As argued in \cite[4.4]{Naoi}, following \cite[13.3.1]{Na}, the
homology $\bigoplus_{\mu}H_*^{BM}(\fM^\la_\mu)$ is cyclic as a
$\fg[t]$-module and generated by
fundamental class of $\fM^\la_\la$.  Dually, this means that
$H^*(\fM^\la_\mu)$ is cocyclic, and cogenerated by the identity in
$H^*(\fM^\la_\la)\cong \C$.  Thus, any map which is compatible with
$\fg[t]$ either kills this vector or is injective.  The former is clearly
false, since $\phi$ in weight $\la$ is an algebra map between copies
of $\C$.  Thus $\phi$ is injective.  
\excise{
The image of $\phi$ obviously contains the identity in $Z(R^\la_\mu)$
for any $\mu$, so if Kirwan surjectivity holds, then $\phi$ must be
surjective, and is thus an isomorphism.}
\end{proof}

\subsection{Surjectivity}
\label{sec:surjectivity}

Let us consider how the maps of Theorem
\ref{thm:final} interacts with Kirwan surjectivity.  For each $\Bv$,
we have a Kirwan map $\kappa_g\colon H^*_{G_\Bv}(*)\to H^*(\fM^\la_{\mu})$; the image of
this map is the same as the subring generated by 
  the Chern classes of the tautological bundles
$V_i$.  Note that we have a natural map $\rho\colon H^*_{G_\Bv}(*)\to
Z(R_{\la-\mu})$ to the center of the KLR algebra; this induced by the
identification of $\operatorname{VV} \colon R_{\la-\mu}\to {}^{\delta}\mathbf{Z}_{\Bv}$ with a convolution algebra
${}^{\delta}\mathbf{Z}_{\Bv}$ in
\cite[3.6]{VV}.  Explicitly, this sends the $k$th Chern class $c_{i,k}$ of the
associated  bundle $V_i\times^{G_{\Bv}} EG_{\Bv}$ to the degree $k$
elementary symmetric polynomial on the strands with label $i$.  Note
that by \cite[2.9]{KLI}, this means that $\rho$ is an isomorphism.
\begin{proposition}
  The following diagram commutes:
\[\tikz[->,very thick]{
\matrix[row sep=10mm,column sep=15mm,ampersand replacement=\&]{
\node (a) {$H^*_{G_\Bv}(*)$}; \& \node (c) {$Z(R_{\la-\mu})$};\\
\node (b) {$H^*(\fM^\la_\mu)$}; \& \node (d) {$Z({{R}}^\la_\mu)$};\\
};
\draw (a) --node[midway,left]{$\kappa_g$}  (b);
\draw (a) -- node[midway,above]{$\rho$} (c);
\draw (b) --node[midway,below]{$\phi$}  (d);
\draw (c)--node[midway,right]{$\kappa_a$}  (d);
}\]
\end{proposition}
\begin{proof}
  This commutation is essentially automatic from the definition.   We
  can fill in a middle term in this commutative diagram:
\[\tikz[->,very thick]{
\matrix[row sep=10mm,column sep=20mm,ampersand replacement=\&]{
\node (a) {$H^*_{G_\Bv}(*)$}; \& \node (c) {$Z({}^{\delta}\mathbf{Z}_{\Bv})$}; \&\node (e) {$Z(R_{\la-\mu})$};\\
\node (b) {$H^*(\fM^\la_\mu)$}; \& \node (d) {$Z(\End(C_\mu))$};\& \node (f) {$Z({{R}}^\la_\mu)$};\\
};
\draw (a) --node[midway,left]{$\kappa_g$}  (b);
\draw (a) -- node[midway,above]{$\rho'$} (c);
\draw (b) --node[midway,below]{$\phi'$}  (d);
\draw (e) -- node[midway,above]{$Z(\operatorname{VV})$}  (c);
\draw (f) -- (d);
\draw (c)--(d);
\draw (e)--node[midway,right]{$\kappa_a$} (f);
}\]
The map $R^\la_\mu\to \End(C_\mu)$ is precisely defined by taking the
image of $\operatorname{VV}$ under a reduction functor, so the right
hand square is commutative by definition.  The left hand square can
thought of as follows.  We can identify ${}^{\delta}\mathbf{Z}_{\Bv}$
with the endomorphisms of the sheaf $L_{\Bv}:=\bigoplus_{|\Bi|=\Bv}L\star\tilde{\mathscr{E}_{i_1}}\star \cdots \star
    \tilde{\mathscr{E}}_{i_n}$ on $T^*X^\la_\mu$ where $L$ is the
    pushforward of functions on the subspace $\oplus_{i}\Hom(V_i,W_i)\subset
    E^\la_\mu$; taking solution sheaf, we obtain the pullback of
    Varagnolo and Vasserot's sheaf ${}^{\delta\!}\mathcal{L}_{\Bi}$
    via the projection $E^\la_\mu\to E^0_{\mu-\la}$, so the
    endomorphism are the same.  

On the other hand, $C_\mu$ is the
    image of $L_\Bv$ under the Kirwan functor, with this functor
    inducing the middle vertical map.  In this context, we can think
    of the Kirwan functor as pullback from $T^*X^\la_\mu $ to
    $\fM^\la_\mu$ of a microlocal D-module.  Thus, we can think of the map
    $\rho'$ as the action induced by the identification $
    H^*(T^*X^\la_\mu)\cong H^*_{G_\Bv}(*)$, induced the natural map of sheaves
    $\C_{T^*X^\la_\mu}\to \mathcal{E}nd( \boldsymbol{\mu} L_{\Bv})$.   

Since both the left and
    middle vertical maps are induced by pullback $\fM^\la_\mu \subset
    T^*X^\la_\mu $, and the maps $\rho',\phi'$ by the map from Betti cohomology
    to Hochschild cohomology, the square commutes by naturality of
    this map.
\end{proof}
Thus, under $\phi$, the images of $\kappa_g$ and $\kappa_a$ coincide.
\begin{corollary}
  If algebraic Kirwan surjectivity (i.e. the surjectivity of
  $\kappa_a$) holds, then the map $\phi$ is an isomorphism
  and the map  $\kappa_g$ is surjective.  That is, the conventional Kirwan
  surjectivity theorem holds for $\fM^\la_\mu$.
\end{corollary}
Combining with Proposition \ref{prop:Kirwan} completes the proof of
Theorem \ref{main}.
It feels slightly embarrassing to the author to write this
as a corollary, since it very easily seen directly from \cite{Naoi}
in the finite case, but for the sake of posterity, let us note:
\begin{corollary}\label{cor:kirwan}
  Kirwan surjectivity holds for quiver varieties of finite ADE.
\end{corollary}
Unfortunately, it's not clear that this algebraic
analog of Kirwan surjectivity is any easier to prove in more general
situations.

\subsection{Generalizations}
\label{sec:equivariant-case}

We can give a small generalization of this theorem: there is a
generalization of $\mathcal{C}_\mu$ to a category $\mathcal{O}$ over
the quiver variety $\fM^\la_\mu$ based on the choice of a Hamiltonian $\C^*$-action.  In \cite[5.23]{BLPWgco}, we
conjecture that for any semi-simple generator $O$ of this category,
the map $H^*(\fM^\la_\mu)\to Z(\Ext^\bullet(O,O))$ is an isomorphism.
We can now show an important special case of this conjecture:
\begin{proposition}
  If $\Gamma$ is a finite ADE quiver and $O$ a
  semi-simple generator of category $\cO$ for any $\C^*$-action, then
  the map $H^*(\fM^\la_\mu)\to Z(\Ext^\bullet(O,O))$ is an isomorphism
\end{proposition}
\begin{proof}
 By \cite[Th. \ref{O-main1}]{Webqui}, we have that $\Ext^\bullet(O,O)$ is Morita
 equivalent to a tensor product algebra $T^\bla$ in finite type.  In both cases,
 the modules over these algebras are covers of $R^\la_\mu\mmod$ by
 \cite[Th. \ref{m-properties}]{Webmerged}; that
 is, the projective modules over $\Ext^\bullet(O,O)$ embed as a
 subcategory of $R^\la_\mu\mmod$.  This implies that these
 algebras have the same center, completing the proof.
\end{proof}

  Theorem \ref{thm:final} can also be upgraded to a
  $K=G_{\Bw}$-equivariant version in finite type, and
  $K=G_{\Bw}\times \C^*$-equivariant with the $\C^*$ acting by weight 1 on the
  edges of the cycle in affine type A.  Let $\check{R}^\la_\mu$ be
  the deformed KLR algebra defined in
  \cite[\S\ref{m-sec:univ-quant}]{Webmerged}; if $\Gamma$ has affine
  type A, then let $\check{\mathbb{R}}^\la_\mu$ be the deformed KLR
  algebra over $\C[h]$ attached to the polynomials \[Q_{ij}^h(u,v)=
(u-v+h)^{\#\{ j\to i\}}
(v-u+h) ^{\#\{ i\to j\}}. \]
  \begin{proposition}
   There is an injective algebra map
   \[H^*_{G_{\mathbf{w}}}(\fM^\la_\mu)\to Z(\check{R}^\la_\mu)\cong  \Ext^\bullet_{G_{\mathbf{w}}}(O,O)\] which
   is compatible with the current algebra actions.  This
   map is an isomorphism for finite type ADE.  In affine
   type A, we can deform further to an isomorphism
   \[H^*_{G_{\mathbf{w}}\times \C^*}(\fM^\la_\mu)\to
   Z(\check{\mathbb{R}}^\la_\mu) \cong
   \Ext^\bullet_{G_{\mathbf{w}}\times \C^*}(O,O).\]
  \end{proposition}
  \begin{proof}
    First, we need to that these maps exist.  We have a natural map
    from $H^*_{K}(\fM^\la_\mu)$ to the Ext algebra of $C_\mu$ in the
  $K$-equivariant derived category.  This equivariant Ext-algebra
  must be a base extension of $\check{R}^\la_\mu$ by
  \cite[\ref{m-universal}]{Webmerged};  this base extension can
  calculated using the minimal polynomial of $y$ acting on
  $C_{\la-\al_i}$;  since the quiver variety in thise case in
  $T^*\mathbb{P}(W_i)$, we have that the only core module can be
  identified with the zero-section, and its equivariant Ext-algebra is
  just \[H^*_{G_{\Bw}}(\mathbb{P}(W_i))\cong \C[y]/(y^{w_i}-c_{i,1}
  y^{w_i-1}+\cdots +(-1)^{w_i}c_{i,w_i})\] where $c_{i,j}$ is the $j$th
  Chern class of the tautological bundle $W_i$ on $*/G_{\Bw}$.  Thus,
  we obtain the universal quantization itself, with the deformation
  parameters precisely matching the Chern classes on $*/G_{\Bw}$.  

In order to extend to the affine deformation, we simply note that the
KLR algebra is deformed to that for $Q_{ij}^h(u,v)$ by the addition of
$\C^*$-action; otherwise, the proof is the same.
  \end{proof}

 \bibliography{../gen}
\bibliographystyle{amsalpha}
\end{document}